\documentclass[12pt]{article}

\usepackage{booktabs}
\usepackage{dsfont}
\usepackage{pifont}
\usepackage{amssymb}
\usepackage{amsmath}
\usepackage{amsthm}
\usepackage{tikz-cd}
\usepackage{stmaryrd} % required for \rbag & \Rbag
\usepackage{arydshln} % for dashed lines in table 
\usepackage[margin=2.4cm]{geometry}

\usepackage[colorlinks=true]{hyperref}
\hypersetup{allcolors=[rgb]{0.1,0.1,0.4}}

\author{Peter Kristel and Benedikt Peterseim} %Authors listed alphabetically

\title{A Topologically Enriched Probability Monad on the Cartesian Closed Category of CGWH Spaces}

\newcommand{\R}{\mathbb{R}}
\newcommand{\K}{\mathbb{K}}
\newcommand{\Ba}{\mathcal{B}a} % Baire sigma algebra
\newcommand{\xmark}{\ding{55}} % X mark
\newcommand{\cmark}{\ding{51}} % check mark

\newcommand{\mc}[1]{\mathcal{#1}}
\newcommand{\1}{\mathds{1}}

\newcommand{\cgwh}{\mathsf{CGWH}} % Category of CGWH spaces
\newcommand{\cg}{\mathsf{CG}} % Category of CG spaces

\newtheorem{theorem}{Theorem}[section]
\newtheorem{remark}[theorem]{Remark}
\newtheorem{warning}[theorem]{Warning}
\newtheorem{lemma}[theorem]{Lemma}
\newtheorem{corollary}[theorem]{Corollary}
\newtheorem{proposition}[theorem]{Proposition}
\newtheorem*{representationtheorem}{Representation Theorem}

\theoremstyle{definition}
\newtheorem{definition}[theorem]{Definition}
\theoremstyle{definition}
\newtheorem{example}[theorem]{Example}

\begin{document}

\maketitle
\begin{abstract}
    Probability monads on categories of topological spaces are classical objects of study in the categorical approach to probability theory, with important applications in the semantics of probabilistic programming languages. We construct a probability monad on the category of compactly generated weakly Hausdorff (CGWH) spaces, a (if not \emph{the}) standard choice of convenient category of topological spaces. 
    Because a general version of the Riesz representation theorem adapted to this setting plays a fundamental role in our construction, we name it the \emph{Riesz probability monad}.
    We show that the Riesz probability monad is a simultaneous extension of the classical Radon and Giry monads that is topologically enriched. Topological enrichment corresponds to a strengthened \emph{continuous mapping theorem} (in the sense of probability theory). 
    In addition, restricting the Riesz probability monad to the Cartesian closed subcategory of weakly Hausdorff quotients of countably based (QCB) spaces results in a probability monad which is strongly affine, ensuring that the notions of independence and determinism interact as we would expect. 
\end{abstract}

%AMS2020: Primary 60B05; Secondary 18C15, 18F60, 54D50

\setcounter{tocdepth}{2}
\tableofcontents

\section{Introduction and results}\label{sec-Introduction}

\subsection{Background: probability monads} Probability monads are one of the centrepieces of category-theoretically informed approaches to probability theory. 
As such, they have found important applications in the context of probabilistic (features of) programming languages \cite{jones1989probabilistic,heunen2017convenient,vakar2019domain}. 
The fundamental idea of a probability monad is that forming the space of probability measures should yield a monad on a suitable base category of ``sample spaces''. 
This was first proposed by Giry in \cite{giry1981categorical}, who constructed two probability monads: the Giry monad on the category $\mathsf{Meas}$ of measurable spaces and a probability monad on the category $\mathsf{Pol}$ of Polish spaces (i.e.~completely metrisable, separable topological spaces), also referred to as the Giry monad. 
Since then, many variations of this idea have been studied \cite{keimel2008monad,fritz2019aprobability,fritz2021probability,van2021probability}; see also \cite{nlab2024monads} for a quick overview. 

The present work is partially based on the fourth chapter of the second author's Master's thesis \cite{peterseim2024monadic}, with a simplified construction,  streamlined towards the results we now present.

\subsection{Overview and main results} In this paper, we consider the category of compactly generated weakly Hausdorff spaces, or CGWH spaces for short (see Definition \ref{def:hk-space}).
\subsubsection{The Riesz probability monad} We construct a monad, $(\mc{P},\delta,\Rbag)$, on this category with the property that, for any CGWH space $X$, the set underlying $\mc{P}(X)$ is the set of $k$-regular (see Definition \ref{definition:kRegularMeasures}) Baire probability measures on $X$.
The component of the natural transformation $\delta: \1 \Rightarrow \mc{P}$ at $X$ is the map that sends $x \in X$ to the Dirac measure supported at $x$. The multiplication $\Rbag: \mc{P}^{2} \Rightarrow \mc{P}$ is given at $X$ conceptually as the $\mc{P}(X)$-valued valued integral, 
    $$ \Rbag: \mc{P}(\mc{P}(X)) \to \mc{P}(X), \;\; \pi \mapsto \int_{\mc{P}(X)} \mu\: \mathrm{d} \pi(\mu). $$
We will give a precise definition in Section \ref{sec:TheMonad}.
As our construction is based on a version of Riesz representation theorem, we call $\mc{P}$ the \emph{Riesz probability monad}. 
\subsubsection{The Baire probability monad} An important Cartesian closed subcategory of $\cgwh$ is the category $\mathsf{QCB}_{h}$ of weakly Hausdorff quotients of countably based (QCB) spaces (Definition \ref{definition:QCBspaces}). The Riesz probability monad $\mc{P}$ restricts to a commutative enriched monad on $\mathsf{QCB}_{h}$, and whenever $X$ is a QCB space, $\mc{P}(X)$ consists exactly of the Baire measures on $X$. For this reason, we call the resulting monad $\mc{P}$ on $\mathsf{QCB}_{h}$ the \emph{Baire probability monad}. (Note that this monad is distinct from the ``Baire monad'' considered in \cite{van2021probability}.)  
\subsubsection{The Baire probability monad is strongly affine} Weakly Hausdorff QCB spaces do not only form a well-behaved category, they are also well-behaved objects in that they share many countability properties with Polish spaces (see Proposition \ref{proposition:propertiesOfQCBspaces}). One consequence of this is that product measures on QCB spaces are particularly simple to understand in this situation (see Lemma \ref{lemma:productMeasuresOnQCBspaces}), from which we deduce that the Baire probability monad on $\mathsf{QCB}_{h}$ is \emph{strongly affine} (see Section \ref{sec:BaireMonadStronglyAffine} for the definition and significance of this property). 
\subsubsection{Relation to classical probability monads} The Baire probability monad on $\mathsf{QCB}_{h}$ further restricts to the classical Giry monad on the category $\mathsf{Pol}$ on Polish spaces (see Section \ref{sec:relationToGiryMonad}). Likewise, the Riesz probability monad restricts to the well-known Radon monad on the category $\mathsf{CompHaus}$ of compact Hausdorff spaces (see Section \ref{sec:relationToRadonMonad}). The situation is summarised in Figure \ref{figure:monadOverview}.

\begin{figure}[ht]\label{figure:monadOverview}
\begin{equation*} 
\begin{tikzcd}
	& {\mathsf{CGWH}}\ar[loop, in=60,out=120,looseness=5]{}{\text{``Riesz monad''}} \\
	{\mathsf{QCB}_{h}}\ar[loop, in=60,out=120,looseness=5, "\text{``Baire monad''}"{anchor=south}] && {\mathsf{CompHaus}}\ar[loop, in=-60,out=-120,looseness=5]{}{\qquad \qquad\qquad\quad\text{``Radon monad''}} \\
	{\mathsf{Pol}}
	\arrow[hook, from=2-1, to=1-2]
	\arrow[hook', from=2-3, to=1-2]
	\arrow[hook, from=3-1, to=2-1]
    \ar[loop, in=190,out=-100,looseness=5]{}{\text{``Giry monad''}}
\end{tikzcd}
\end{equation*}\caption{Relations between different probability monads. The hooked arrows are fully faithful functors. Each of these can be (trivially) extended to a morphism of monads.}
\end{figure}
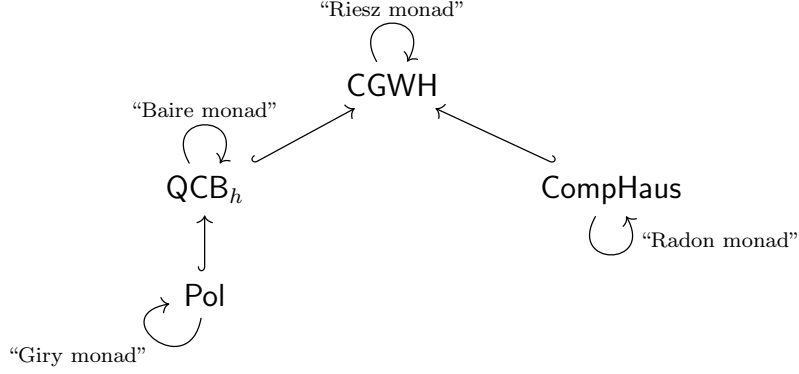

\subsubsection{The role of Cartesian closedness} There are several reasons why we choose $\cgwh$, along with its smaller variant $\mathsf{QCB}_{h}$, as our base categories. 
First, CGWH spaces serve as a common Cartesian closed substitute for topological spaces, which raises the question whether the well-known probability monads defined on categories of topological spaces \cite{van2021probability,fritz2021probability} admit an analogue on $\cgwh$.
Secondly, the fact that $\cgwh$ and $\mathsf{QCB}_{h}$ are Cartesian closed makes them suitable for treating the semantics of \emph{higher-order probabilistic programming languages}, as done in \cite{heunen2017convenient}, where the category $\mathsf{QBS}$ of \emph{quasi-Borel spaces} is employed instead. 
In contrast to quasi-Borel spaces, which were introduced as a Cartesian closed replacement for the category of measurable spaces, CGWH spaces are much more well-studied objects. 
In this sense, our choice of base category is more straightforward. 
Moreover, an important feature distinguishing the Baire probability monad on $\mathsf{QCB}_{h}$ from the probability monad on quasi-Borel spaces is that it is \emph{strongly affine}. 
Failure of the strongly affine property in the setting of quasi-Borel spaces leads to somewhat counter-intuitive phenomena, which we are thus able to avoid; see Section \ref{sec:BaireMonadStronglyAffine} for details. We do not know whether the Riesz probability monad is also strongly affine.

\subsubsection{Topological enrichment} Finally, the Riesz and Baire probability monads are enriched monads (over their respective base categories), and since we are working with categories of topological spaces, this means that they are \emph{topologically enriched}. Topological enrichment has a natural probabilistic interpretation as a \emph{strengthened continuous mapping theorem} (see Section \ref{sec:EnrichedStructure}). Table \ref{table:MonadComparison} gives an overview of several probability monads which have been proposed and how ours differs from these.

\begin{table}[ht]
\centering
\begin{tabular}{@{}llcccc@{}}
\toprule
Name  & Base Category       & Type of Measure                 & \begin{tabular}[c]{@{}c@{}}Strongly\\ Affine?\end{tabular} & \begin{tabular}[c]{@{}c@{}}Enriched \\ over CCC?\end{tabular} & Reference                   \\ \midrule
Giry  & $\mathsf{Meas}$     & Any                             & \cmark                                                     & \xmark                                                        & \cite{giry1981categorical}  \\
--    & $\mathsf{QBS}$      & see \cite{heunen2017convenient} & \xmark                                                     & \cmark                                                        & \cite{heunen2017convenient} \\
Giry  & $\mathsf{Pol}$      & Borel                           & \cmark                                                     & \xmark                                                        & \cite{giry1981categorical}  \\
Radon & $\mathsf{CompHaus}$ & Radon                           & \cmark                                                     & \xmark                                                        & \cite{swirszcz1974monadic}  \\
Riesz & $\mathsf{CGWH}$     & $k$-regular                     & \textbf{?}                                                          & \cmark                                                        & (this paper)                \\
Baire & $\mathsf{QCB}_{h}$     & Baire                           & \cmark                                                     & \cmark                                                        & (this paper)                \\ \bottomrule
\end{tabular}
\caption{Comparison of the Riesz and Baire probability monad to other examples from the literature. The Baire probability monad is the only one that is both enriched over a Cartesian closed base category (CCC) and known to be strongly affine.}\label{table:MonadComparison}
\end{table}

\subsection{Outline of the construction} Our construction of the Riesz and Baire probability monads can briefly be summarised as follows. Throughout this paper, $\mathbb{K}$ will denote the field of real or complex numbers.

\subsubsection{$k$-regular Baire measures} We first define the set $\mc{M}(X)$ of $\mathbb{K}$-valued \emph{$k$-regular Baire measures} on a CGWH space $X$ (Definition \ref{definition:kRegularMeasures}). When $X$ is a QCB space, every $\mathbb{K}$-valued Baire measure is $k$-regular (Lemma \ref{lemma:radonMeasuresKregular}), so $\mc{M}(X)$ is simply the set of Baire measures on $X$.

\subsubsection{A Riesz representation theorem} We then prove a Riesz representation theorem (Theorem \ref{theorem:RieszForCbAndM}), 
    $$ \mc{M}(X) = C_b(X)', $$
identifying $\mc{M}(X)$ with the continuous dual space $C_b(X)'$ of the space $C_b(X)$ of continuous bounded functions on $X$, equipped with its natural CGWH topology (Definition \ref{definition:CbCGWHSpace}). This natural CGWH topology on $C_b(X)$ is \emph{not} the Banach space topology, unless $X$ is compact (see Section \ref{section:CbAsLinearCGWHspace} for further motivation regarding this choice of topology).

\subsubsection{The endofunctor $\mc{M}$} The space $C_b(X)$ is an example of a \emph{linear CGWH space} (Definition \ref{definition:LinearCGWHspace}), a notion similar to, but generally distinct from, that of a topological vector space.
Due to the Cartesian closedness of $\cgwh$, the continuous dual space $V'$ of a linear CGWH space $V$ acquires a natural CGWH topology, under which it becomes the \emph{natural dual} $V^\wedge$ of $V$ (Definition \ref{definition:naturalDual}). Thus, under the identification $\mc{M}(X)=C_b(X)^\wedge$, the space of $k$-regular measures becomes a CGWH space and $\mc{M}$ an endofunctor on $\cgwh$. 

\subsubsection{The monad structure} The endofunctor $\mc{M}=C_b(-)^\wedge$ now being given by a double-dualisation-type operation allows us to endow it with a $\cgwh$-enriched monad structure in a seamless way. The resulting \emph{Riesz monad} $\mc{M}$ restricts to a (commutative enriched) monad on the category $\mathsf{QCB}_{h}$ of weakly Hausdorff QCB spaces. Since $\mc{M}(X)$ consists exactly of the Baire measures on $X$ when $X$ is a QCB space, we call the resulting monad the \emph{Baire monad}. 

\subsubsection{Passing to probability measures} Finally, passing to the subspace $$\mc{P}(X) \subseteq \mc{M}(X) $$ 
of probability measures yields the Riesz and Baire \emph{probability} monads on $\cgwh$ and $\mathsf{QCB}_{h}$, respectively (Theorem \ref{theorem:rieszProbabilityMonad}).

\subsection{A monads-probability dictionary} Probability monads clarify various \emph{structural} aspects of probability: there is an interesting correspondence between properties of (or structure on) monads and certain phenomena in probability theory, summarised in Table \ref{table:monadProbabilityDictionary}.
Our addition to this dictionary is the role of topological enrichment.

\begin{table}[htp]
\centering
\begin{tabular}{@{}ccc@{}}
\toprule
\textbf{Monads}                                                                       & \textbf{Probability Theory}                                                                                                 & \begin{tabular}[c]{@{}c@{}}\textbf{Section/}\\ \textbf{Reference}\end{tabular} \\[.2cm] \midrule
unit                                                                                  & \begin{tabular}[c]{@{}c@{}}Dirac measures/\\ deterministic random variables\end{tabular}                                    & \ref{sec:passingToPX}                                                          \\[.4cm] 
multiplication                                                                        & \begin{tabular}[c]{@{}c@{}}expectation of random \\ probability measure\end{tabular}                                        & \ref{sec:passingToPX}                                                          \\[.4cm] 
Kleisli morphism                                                                      & Markov kernel                                                                                                               & \cite{fritz2020synthetic}                                                      \\[.2cm] 
\begin{tabular}[c]{@{}c@{}}symmetric monoidal structure/\\ commutativity\end{tabular} & \begin{tabular}[c]{@{}c@{}}product measures/\\ Fubini's theorem\end{tabular}                                                & \ref{sec:commutativityAndFubini}                                               \\[.6cm] 
strongly affine property                                                              & \begin{tabular}[c]{@{}c@{}}independence of deterministic \\ random variables \\ (of any other random variable)\end{tabular} & \ref{sec:BaireMonadStronglyAffine}                                             \\[.7cm] 
\emph{topological enrichment}                                                         & \begin{tabular}[c]{@{}c@{}}\emph{strengthened}\\ \emph{continuous mapping theorem}\end{tabular}                            & \ref{sec:EnrichedStructure}                                                    \\[.4cm]  \bottomrule
\end{tabular}
\caption{A monads--probability dictionary.}\label{table:monadProbabilityDictionary}
\end{table}

\section{Compactly generated weakly Hausdorff (CGWH) spaces}

In this section, we give an exposition of the category of CGWH spaces.
A thorough account is given in \cite{strickland2009category}.

\begin{warning}\label{warn:CGspaceDefinitions}
    In the literature, the term ``compactly generated (CG) space'', and its occasional synonym ``$k$-space'', have various (slightly different) meanings. 
    Our conventions agree with \cite{strickland2009category}. For clarity, we give all relevant definitions below.
\end{warning}

\begin{definition}\label{def:CGTopology}
    Let $X$ be a set.
    A topology on $X$ is \emph{compactly generated}, or \emph{CG}, if a subset $U \subseteq X$ is open if and only if $f^{-1}(U) \subseteq K$ is open for all continuous maps $f: K \rightarrow X$, where $K$ is compact Hausdorff.

    If $X$ is equipped with a CG topology, then $X$ is said to be a \emph{compactly generated space}, or \emph{CG space}. We write $\cg$ for the category of CG spaces with continuous maps as morphisms.
\end{definition}

Let $\tau$ be a topology on a set $X$.
We say that $U\subseteq X$ is $k(\tau)$-open if $f^{-1}(U) \subseteq K$ is open in $K$ for all compact Hausdorff spaces $K$ and all continuous maps $f:K \rightarrow X$.
The following fundamental lemma is easily verified directly.
\begin{lemma}\label{lem:kificiation}
    The collection of $k(\tau)$-open sets forms a topology on $X$.
    Additionally, every $\tau$-open set is $k(\tau)$-open.
\end{lemma}

%\begin{proof}
%    The statement that every $\tau$-open set is $k(\tau)$-open is obvious.
%    Let $U_{i}$ be a finite collection of $k(\tau)$-open sets.
%    Let $K$ be an arbitrary $\tau$-compact subset.
%    We have that $U_{i} \cap K$ is open in $K$, and thus $(\cap_{i} U_{i}) \cap K$ is open in $K$, which implies that $\cap_{i} U_{i}$ is $k(\tau)$-open.
%    Showing that unions of arbitrary families of $k(\tau)$-open subsets are also open can be done in a similar fashion.
%\end{proof}

Observe that a topological space $(X,\tau)$ is CG if and only if $k(\tau) = \tau$. The space $kX := (X, k(\tau))$ is called the \emph{$k$-ification} of $(X, \tau)$.

The following result shows that the $k$-operator on topologies is idempotent.
\begin{lemma}\label{lemma:kOperatorIdempotent}
    Let $K$ be a compact Hausdorff space.
    A map $f: K \rightarrow X$ is $\tau$-continuous, if and only if, it is $k(\tau)$-continuous.
    Thus $k(k(\tau)) = k(\tau)$, and $(X,k(\tau))$ is a CG space.
\end{lemma}
\begin{proof}
    If $f:K \rightarrow X$ is $k(\tau)$-continuous, then it is also $\tau$-continuous, because every $\tau$-open set is $k(\tau)$-open.

    Now, suppose that $f:K \rightarrow X$ is $\tau$-continuous.
    It follows immediately from the definition that if $U$ is $k(\tau)$-open, then $f^{-1}(U)$ is open, whence $f$ is $k(\tau)$-continuous.
\end{proof}

%\begin{lemma}
%    Suppose that every $\tau$-compact subset of $X$ is $\tau$-
%    Hausdorff. 
%    A subset $K\subset X$ is $\tau$-compact, if and only if, it is $      k(\tau)$-compact.
%    Thus $k(k(\tau)) = k(\tau)$ and $(X,k(\tau))$ is a CG space.
%\end{lemma}
%\begin{proof}
%    Let $K \subset X$ be $\tau$-compact.
%    Let $\{U_{i}\}_{i \in I}$ be a $k(\tau)$-open cover of $K$.
%    We obtain that $\{ U_{i} \cap K\}_{i \in I}$ is an open cover of $K$, in the subspace topology of $K$. This admits a finite subcover, $\{ U_{i} \cap K \}_{i \in J}$ by compactness of $K$.
%    We therefore obtain a finite family $\{ U_{i} \}_{i \in J}$ of $k(\tau)$-open sets which covers $K$, whence $K$ is $k(\tau)$-compact.

%    Conversely, let $K \subset X$ be $k(\tau)$-compact.
%    Let $\{ U_{i} \}_{i \in I}$ be a $\tau$-open cover.
%    Because every $\tau$-open set is $k(\tau)$-open, this family admits a finite refinement. Thus, $K$ is $\tau$-compact.
%\end{proof}

In addition to idempotency, we also have monotonicity.

\begin{lemma}\label{lemma:kOperatorMonotone}
    Let $X$ be a set and let $\tau_1 \subseteq \tau_2$ be two topologies on $X$. Then $k(\tau_1) \subseteq k(\tau_2)$
\end{lemma}
\begin{proof}
    Let $U\in k(\tau_1)$. Let $K$ be a compact Hausdorff space and let $g: K \to (X, \tau_2)$ be continuous. Then $g$ is also continuous as a map $K \to (X, \tau_1)$. By definition of $k(\tau_1)$, this implies that $g^{-1}(U)$ is open in $K$. Since $K$ and $g$ were arbitrary, we conclude that $U\in k(\tau_2)$.
\end{proof}

\subsection{Limits and colimits of CG spaces} 

We now show that $\cg$ is bicomplete. First, we describe colimits in $\cg$.

Let $X$ be a set, and let $\{ Y_{i} \}_{i \in I}$ be a family of CG spaces.

\begin{lemma}\label{lemma:finalTopologyCG}
    Let $\{f_{i} \}_{i \in I}$ be a family of maps $f_{i}:Y_{i} \rightarrow X$. Let $\tau$ be the final topology with respect to $\{f_{i} \}$. Then $\tau$ is a CG topology.
\end{lemma}
\begin{proof}
    The final topology $\tau$ with respect to $\{f_{i} \}_{i \in I}$ is the finest topology such that all maps $f_{i}$ are continuous.
    By Lemma \ref{lem:kificiation} we have that $k(\tau)$ is at least as fine as $\tau$.
    Thus, it suffices to prove that all maps $f_{i}$ are continuous with respect to $k(\tau)$.

    So, let $U \subseteq X$ be $k(\tau)$-open.
    Let $K$ be a compact Hausdorff space, and $f: K \rightarrow Y_{i}$ a continuous map.
    We have that $f_{i} \circ f: K \rightarrow X$ is continuous, so $f^{-1}f_{i}^{-1}(U)$ is open.
    Because $Y_{i}$ is CG this implies that $f_{i}^{-1}(U)\subseteq Y_{i}$ is open.
\end{proof}

Given Lemma \ref{lemma:finalTopologyCG}, it follows that to obtain a colimit in $\cg$ one simply takes the colimit in the category of topological spaces, and observes that the resulting topological space is CG.

Next, we describe limits in $\cg$ in a dual manner.
\begin{definition}\label{def:initialkTop}
    Let $\{f_{i}\}_{i \in I}$ be a family of maps $f_{i}: X \rightarrow Y_{i}$.
    Let $\tau$ be the initial topology with respect to $\{ f_{i} \}_{i \in I}$.
    The \emph{initial CG-topology} on $X$ with respect to $\{f_{i}\}_{i \in I}$ is $k(\tau)$.
\end{definition}

It's clear that the initial CG-topology is CG. In fact, more is true. 

\begin{lemma}\label{lemma:charOfInitialCGtopology}
    The initial CG-topology with respect to the family of maps $f_{i}: X \rightarrow Y_{i}$ is the coarsest CG-topology making all the $f_i$ continuous.
\end{lemma}
\begin{proof}
    The initial topology $\tau$ is the coarsest topology making all the $f_i$ continuous. So if $\tau'\subseteq k(\tau)$ is any CG-topology making all $f_i$ continuous, then $\tau \subseteq \tau'$ and since $\tau'$ is a CG-topology, $k(\tau)\subseteq k(\tau') = \tau'$ (see Lemmas \ref{lemma:kOperatorIdempotent} and \ref{lemma:kOperatorMonotone}). Therefore, $k(\tau)= \tau'$ and $k(\tau)$ is indeed the coarsest topology making all the $f_i$ continuous. 
\end{proof}

Using this, we describe limits in $\cg$ as follows.

\begin{lemma}\label{lem:LimitsInCG}
    Let $X_\bullet: I \to \cg$ be a diagram of CG spaces.
    Let $|X|$ be the limit of $|X_\bullet|$ as taken in the category of sets (where $|\cdot|: \cg \to \mathsf{Set}$ is the forgetful functor), and let $\tau$ be the initial CG-topology with respect to the family of maps $f_i: |X| \to |X_i|$ (the universal cone).
    Then the space $X=(|X|,\tau)$ is the limit of $X_\bullet$ in $\cg$.
\end{lemma}
\begin{proof}
    Let $Y$ be any CG space and let $g_i: Y \to X_i$ ($i\in I$) be a cone over $X_\bullet$. By the universal property of $|X|$ as the limit of $|X_\bullet|$ in $\mathsf{Set}$, there exists a unique map $u: |Y| \to |X|$ such that $g_i = f_i \circ u$ for all $i\in I$. We need to show that $u$ is continuous as a map $Y \to X$. Since $Y$ is CG, the final topology $\tau_u$ with respect to $u$ is a CG topology on $X$ (Lemma \ref{lemma:finalTopologyCG}). To verify continuity of $u$, it suffices to show that $\tau_u$ contains (i.e.~is finer than) the topology of $X$. Because $X$ carries the coarsest CG-topology making all $f_i$ continuous (Lemma \ref{lemma:charOfInitialCGtopology}), this reduces the claim to showing that all $f_i$ are continuous with respect to $\tau_u$. But this is to say that for all $i\in I$ and all open $U\subseteq X_i$, we have that $u^{-1}(f^{-1}(U))$ is open in $Y$, which follows from the continuity of the $g_i$ and our assumption that $g_i= f_i \circ u$. 
\end{proof}

\begin{corollary}\label{corollary:colimitsInCG}
    The category $\cg$ of CG spaces is bicomplete. Limits are formed as described in Lemma \ref{lem:LimitsInCG}, and colimits are formed as in the category of topological spaces, by equipping the colimit of the underlying diagram of sets with the final topology (with respect to the universal co-cone). 
\end{corollary}

\begin{warning}
    As a particular case of the limits of CG spaces constructed in Lemma \ref{lem:LimitsInCG}, the product $X\times Y$ of two CG spaces $X$ and $Y$ does \emph{not} generally coincide with the product $X\times_{\mathsf{Top}} Y$ as taken in the category $\mathsf{Top}$ of topological spaces, see \cite[Example 3.3.29]{engelking1989general}. Similar warnings apply to the various $\sigma$-algebras one may consider on such product spaces; these will play a crucial role later on. \emph{A priori,} $\mathcal{B} a(X \times Y)$, $\mathcal{B} a(X \times_{\mathsf{Top}} Y)$ and $\mathcal{B} a(X) \otimes \mathcal{B} a(Y)$ might \emph{not} agree. However, these $\sigma$-algebras do importantly all coincide for weakly Hausdorff QCB spaces (see Lemma \ref{lemma:productMeasuresOnQCBspaces}).
\end{warning}

\subsection{\texorpdfstring{$\cg$}{CG} is Cartesian closed} Let $X$ and $Y$ be CG spaces.
We now show that there is always an exponential object, denoted either $Y^X$ or $C(X,Y)$, in $\cg$.
The set underlying $Y^X = C(X,Y)$ is the set of continuous maps $X\to Y$. Its topology can be described as follows \cite[Definition 0.1]{booth1980monoidal}.

\begin{definition}\label{definition:compactOpenCGtopology}
    For every compact Hausdorff space $K$, every continuous map $g: K \to X$ and every open subset $U\subseteq Y$, let 
        $$ W(K, g, U) := \{f \in C(X, Y) |\: f(g(K)) \subseteq U \}.  $$
    Let $\tau$ be the topology on $C(X, Y)$ generated by all sets of the form $W(K, g, U)$. We define $Y^{X}$ to be the space of continuous maps equipped with the topology $k(\tau)$, and call $k(\tau)$ the \emph{compact-open CG-topology}. 
\end{definition}

The following theorem is standard and can be found in \cite[Theorem 4.4]{booth1980monoidal}, \cite[Theorem 3.1]{day1972reflection} or \cite[Theorem 3.6]{escardo2004comparing}.

\begin{theorem}
    The category $\cg$ is Cartesian closed. For all CG spaces $X,Y$, the exponential $Y^X$ is given by the space $C(X,Y)$, equipped with the compact-open CG-topology.
\end{theorem}

\subsection{The weak Hausdorff property} The appropriate substitute for the Hausdorff property in the context of $\cg$ spaces turns out to be the following. 

\begin{definition}\label{def:hk-space}
    A topological space $X$ is \emph{weakly Hausdorff} (WH) if every compact subspace of $X$ is Hausdorff.
    
    A CG space that is also weakly Hausdorff is called a \emph{CGWH space}. We write $\cgwh$ for the full subcategory of $\cg$ whose objects are CGWH spaces.
\end{definition}

\begin{remark}
    If $(X,\tau)$ is a WH space, then a subset $U \subseteq X$ is $k(\tau)$-open if, and only if, $U\cap K$ is open in $K$ for all compact subsets $K\subseteq X$.
    Thus, in this case, we do not need to quantify over all continuous maps from compact Hausdorff spaces into $X$; consequently, some of the different definitions mentioned in Warning \ref{warn:CGspaceDefinitions} of CG space collapse.
    Moreover, we should point out that if $X$ is a CGWH space, and $Y$ is a CG space, then the compact-open CG-topology on $Y^{X}$ (Definition \ref{definition:compactOpenCGtopology}) agrees with the $k$-ification of the compact-open topology on $Y^{X}$; a fact that we will use implicitly later.
\end{remark}

A CG space $X$ is weakly Hausdorff if, and only if, the diagonal, 
    $$ (=_X) := \{(x, y) \in X \times X \, | \: x=y \}, $$
is closed in the product $X \times X$ in $\cg$ \cite[Proposition 2.3]{mccord1969classifying}. This is completely analogous to the fact that a topological space $X$ is Hausdorff if, and only if, the diagonal is closed in the $\mathsf{Top}$-product $X\times_{\mathsf{Top}} X$. 

\subsection{The weak Hausdorff quotient} CGWH spaces form an exponential ideal in $\cg$ (Theorem \ref{theorem:CGWHexpIdeal}), by means of the \emph{weak Hausdorff quotient}, or \emph{WH quotient}, for short.

\begin{definition}\label{def:WHQuotient}
    Let $X$ be a CG space. The \emph{WH quotient} of $X$ is $X/\!\sim_{\mathrm{WH}}$, where $\sim_{\mathrm{WH}}$ is the smallest closed (as a subset of $X \times X$) equivalence relation on $X$. (Here, the product is taken in $\cg$, as always.)
\end{definition}

Let us verify that the WH quotient deserves this name.

\begin{lemma}
    Let $X$ be a CG space. The WH quotient $X/\!\sim_{\mathrm{WH}}$ is a CGWH space.
\end{lemma}
\begin{proof}
    The WH quotient $X/\!\sim_{\mathrm{WH}}$ is CG by Lemma \ref{lemma:finalTopologyCG}. It remains to show that $X/\!\sim_{\mathrm{WH}}$ is also WH. By \cite[Proposition 2.3]{mccord1969classifying}, it suffices to verify that the diagonal is closed in $(X/\!\sim_{\mathrm{WH}}) \times (X/\!\sim_{\mathrm{WH}})$. Let $p: X \to X/\!\sim_{\mathrm{WH}}$ be the canonical projection. By \cite[Proposition 2.2]{mccord1969classifying}, the product $p \times p: X \times X \to (X/\!\sim_{\mathrm{WH}}) \times (X/\!\sim_{\mathrm{WH}})$ is a quotient map. Hence, the diagonal is closed in $(X/\!\sim_{\mathrm{WH}}) \times (X/\!\sim_{\mathrm{WH}})$ if, and only if, its preimage in $X\times X$ under $p \times p$ is closed. This preimage is given precisely by $\sim_{\mathrm{WH}}$ which is closed by definition.
\end{proof}

\begin{theorem}\label{theorem:CGWHexpIdeal}
    The category $\cgwh$ is an exponential ideal in $\cg$. This means that $\cgwh$ is a reflective subcategory of $\cg$ and for all CG spaces $X$ and CGWH spaces $Y$, the exponential $Y^X$ is a CGWH space. The left adjoint to the inclusion $\cgwh \hookrightarrow \cg$ is given by the WH quotient.
\end{theorem}
\begin{proof}
    First, we verify that the WH quotient is indeed left adjoint to the inclusion $\cgwh \hookrightarrow \cg$. Let $X$ be a CG space, let $Y$ be a CGWH space and let $f: X \to Y$ be a continuous map. We need to show that $f$ factors uniquely over the canonical projection $p: X\to X/\!\sim_{\mathrm{WH}}$ to the WH quotient. By the universal property of the quotient, it suffices to show that for all $x,y \in X$ with $x \sim_{\mathrm{WH}} y$, we have $f(x)=f(y)$. But this follows from the fact that the equivalence relation, 
        $$ s \sim t \;\;:\Leftrightarrow \;\; f(s)=f(t), $$
    is closed in $X\times X$ (since $Y$ is WH and $f$ is continuous), together with the definition of $\sim_{\mathrm{WH}}$ as the smallest closed equivalence relation on $X$. This completes the proof that the WH quotient is left adjoint to the inclusion $\cgwh \hookrightarrow \cg$.

    Finally, let $X$ be a CG space and let $Y$ be a CGWH space. It remains to show that $Y^X$ is WH. The diagonal of $Y^X$ is given by, 
        $$ (\,=_{Y^X}) \;= \:\{(f, g) \in Y^X \times Y^X \,|\: \forall x \in X, \: f(x)=g(x)\}. $$
    Since point evaluation is continuous and $Y$ is WH, this is an intersection of closed sets, which is hence closed. 
\end{proof}

\begin{corollary}
    The category $\cgwh$ is Cartesian closed and bicomplete. Limits are formed as in $\cg$ (see Lemma \ref{lem:LimitsInCG}) and colimits are formed as the WH quotient of colimits in $\cg$ (see Corollary \ref{corollary:colimitsInCG}).
\end{corollary}

\subsection{Examples of CGWH spaces}\label{sec:Examples}

As we have seen, CGWH spaces are closed under a rich variety of constructions. Complementing this, we give two large classes of examples of CGWH spaces. First, note that compact Hausdorff spaces are trivially CGWH. Moreover: 

\begin{example}
    Closed subspaces, as well as open subspaces of CGWH spaces are again CGWH spaces in the subspace topology \cite[Lemma 2.26]{strickland2009category}.
\end{example}

Since a locally compact Hausdorff space is open in its one-point compactification, we obtain a further class of examples.

\begin{example}
    Locally compact Hausdorff spaces are CGWH spaces.
\end{example}

Recall that a topological space $X$ is said to be \emph{sequential} if every map $f:X \to Y$ to some further topological space $Y$ which is sequentially continuous (i.e.~it sends convergent sequences to convergent sequences) is continuous. Using the observation that convergent sequences can be identified with continuous from $\mathbb{N}\cup \{\infty\}$ (the one-point compactification of the natural numbers), one finds another large class of examples \cite[Proposition 1.6]{strickland2009category}.

\begin{example}
    Every weakly Hausdorff sequential topological space is a CGWH space.
\end{example}

This includes in particular every metrisable space. A further class of weakly Hausdorff sequential spaces, given by weakly Hausdorff \emph{QCB spaces}, will be introduced in Section \ref{sec:QCBspaces} below. These spaces form a Cartesian closed category subcategory of $\cgwh$ particularly well-suited for our cause and deserve a separate treatment.

\subsection{QCB spaces}\label{sec:QCBspaces} QCB spaces were first considered in the context of domain theory and computable analysis  \cite{menni2002topological,schroder2002extended,battenfeld2007convenient}. These spaces find a striking balance between offering a rich categorical structure while simultaneously being amenable to the type of countability arguments often required in measure-theoretic contexts.

\begin{definition}\label{definition:QCBspaces}
    A topological space $X$ is a \emph{QCB space} (short for ``quotient of a countably based space'') if there exists a second-countable space $Y$ together with a quotient map $Y\twoheadrightarrow X$.
    The category of weakly Hausdorff QCB spaces will be denoted by $\mathsf{QCB}_{h}$.
\end{definition}

The following important fact can be found in \cite[Corollary 7.3, Remark 7.4]{escardo2004comparing}, using that CGWH spaces form an exponential ideal in the category of all compactly generated spaces \cite[Proposition 8.10]{rezk2017compactly}.

\begin{theorem}\label{theorem:QCBCartesianClosed}
    The category $\mathsf{QCB}_{h}$ is Cartesian closed and has all countable limits and colimits. These limits and colimits are formed as in $\cgwh$, and the same is true for the exponential objects in $\mathsf{QCB}_{h}$ (i.e. spaces of continuous maps). 
\end{theorem}

In addition, QCB spaces share many of their measure-theoretically desirable countability properties with Polish spaces \cite[Theorem 4.5, Propositions 4.7 and 4.8]{battenfeld2007convenient}.

\begin{proposition}\label{proposition:propertiesOfQCBspaces}
    Every QCB space $X$ is 
    \begin{enumerate}\itemsep-.1em
        \item sequential, 
        \item hereditarily separable, i.e.~every subset of $X$ is separable (in the subspace topology),
        \item and hereditarily Lindelöf, i.e.~for each subset $S\subseteq X$, every cover of $S$ has a countable subcover.
    \end{enumerate}
\end{proposition}

Hence, we may view (weakly Hausdorff) QCB spaces as a natural Cartesian closed generalisation of Polish spaces. \par 
A particular consequence of weakly Hausdorff QCB spaces being closed under countable limits in $\cgwh$ is the following.

\begin{lemma}\label{lemma:closedSubspacesQCB}
    Closed subspaces of weakly Hausdorff QCB spaces are QCB spaces (in the subspace topology). 
\end{lemma}
\begin{proof}
    Let $X$ be a weakly Hausdorff QCB space and let $A\subseteq X$ be a closed subset. Then $A$ is the equaliser of the canonical projection,
        $$ p: X \to X/A, $$
    and the canonical map, 
        $$ *: X \to X/A, \;\; x \mapsto [a], $$
    sending every point in $X$ to the unique equivalence class of any point $a\in A$. Since equalisers are limits and both $X$ and $X/A$ are weakly Hausdorff QCB spaces, where we use use the fact that $A$ is closed to conclude that $X/A$ is weakly Hausdorff (see \cite[Corollary 2.21]{strickland2009category}), the claim follows from Theorem \ref{theorem:QCBCartesianClosed}.
\end{proof}

\subsection{Convergent sequences in mapping spaces}

We will use the following fact concerning the topology of the mapping spaces $Y^X$ to relate the Riesz monad to the Radon monad on compact Hausdorff spaces; see Section \ref{sec:relationToRadonMonad}. 

\begin{lemma}\label{lemma:YtoTheXforYmetricSpace}
    Let $Y$ be a metric space and let $X$ be a compact Hausdorff space. Then $Y^X$ is metrisable and carries the uniform topology. 
\end{lemma}
\begin{proof}
    By definition, the topology on $Y^X$ is given by the $k$-ification of the compact-open topology (see Definition \ref{definition:compactOpenCGtopology}). When $X$ is compact, the latter agrees with the uniform topology, which is metrisable and hence agrees with its $k$-ification.
\end{proof}

The next lemma will be needed in order to understand the convergent sequences in the CGWH space $C_b(X)$; see Lemma \ref{lemma:ConvergentSequencesInCb}.

\begin{lemma}\label{lemma:ConvergenceInYtoTheX}
    Let $Y$ be a metric space and let $X$ be any CGWH space. Then a sequence $(f_n)$ converges in $C(X,Y)=Y^X$ if, and only if, it converges uniformly on compact subsets of $X$.
\end{lemma}
\begin{proof}
    The convergent sequences in a topological space agree with those of its $k$-ification. (This can be seen by identifying convergent sequences with maps from the one-point compactification of the natural numbers.) The topology on $Y^X$ is given by the $k$-ification of the compact-open topology. This is the topology of uniform convergence on compact subsets, which yields the claim.
\end{proof}

\section{Linear CGWH spaces}

Subsequently, we will exclusively consider vector spaces over either the real or complex numbers, which we generically denote by $\mathbb{K}$.
If $X$ is a CGWH space, we denote by $C(X)= C(X, \mathbb{K})= \mathbb{K}^X$ the space of continuous maps into $\mathbb{K}$. 
The space $C(X)$ is a prime example of a \emph{linear CGWH space}.

\begin{definition}\label{definition:LinearCGWHspace}
    A linear CGWH space is a CGWH space $V$ together with two continuous maps (addition and scalar multiplication),
        $$ +: V \times V \to V,$$
        $$ \cdot: \mathbb{K}\times V \to V,$$
    such that $V$ forms a vector space with respect to the operations $\cdot$ and $+$. (As usual, the products are taken in $\cgwh$.) 
\end{definition}

\begin{warning}
    CGWH spaces are, by definition, topological spaces and the definition of a linear CGWH space is very similar to that of a topological vector space. \emph{However,} with respect to the same topology and vector space structure, a linear CGWH space is \emph{not} necessarily a topological vector space \cite[Example 2.2.9]{peterseim2024monadic}. Continuity of addition is only required with respect to the product $\times$ in $\cgwh$, which is in general a weaker requirement than continuity with respect to $\times_{\mathsf{Top}}$. In the other direction, a Hausdorff topological vector space is also not necessarily a linear CGWH space with respect to the same topology and vector space structure, as exemplified by the weak-$*$ dual of an infinite dimensional Hilbert space \cite[Proposition 1]{frolicher1972topologies}. 
\end{warning}

What does hold is that Hausdorff topological vector spaces always become linear CGWH spaces when equipped with the $k$-ification of their topology. In particular, every Hausdorff topological vector space whose underlying topological space is already a CGWH space is a linear CGWH space with respect to the same topology. Importantly, this includes all metrisable topological vector spaces, such as Fréchet and Banach spaces.

\subsection{\texorpdfstring{$C_b(X)$}{Cb(X)} as a linear CGWH space}\label{section:CbAsLinearCGWHspace}

Let $X$ be a CGWH space. We now describe the linear CGWH space structure on the space $C_b(X)$ of continuous bounded ($\mathbb{K}$-valued) functions on $X$. While the Banach space topology on $C_b(X)$ does yield such structure, this is \emph{not} the topology on $C_b(X)$ which we will consider, as it generally forgets too much information about the topology of $X$ to possibly be useful in our context. For instance, the dual of the \emph{Banach space} $C_b(X)$ can be identified with the space of certain \emph{finitely} additive set functions on $X$ (see \cite[Theorem 7.9.8]{bogachev2007measure}), whereas we are interested in (suitably regular) countably additive measures. \par 
Fortunately, there is a natural $\cgwh$ topology on $C_b(X)$ due to Cartesian closedness of this category.
To see this, note that, as a set, $C_b(X)$ is the directed union over functions uniformly bounded by some natural number, 
    $$C_b(X) = \bigcup_{n \in\mathbb{N}} (nD)^X, $$
where $D = \{\lambda\in \mathbb{K} \mid |\lambda| \leq 1\} \subseteq \mathbb{K}$ is the unit disk and $(nD)^X$ is the exponential in $\cgwh$. Interpreting this directed union as a sequential colimit leads to the following definition.

\begin{definition}\label{definition:CbCGWHSpace}
    We define the linear CGWH space $C_b(X)$ of bounded continuous ($\mathbb{K}$-valued) functions on a CGWH space $X$ as the filtered colimit (in the category of CGWH spaces),
        $$ C_b(X) := \mathrm{colim}_{n\in\mathbb{N}}\: (nD)^X, $$
    of the diagram,
        $$ D^X \hookrightarrow (2D)^X \hookrightarrow (3D)^X \hookrightarrow ... ,  $$
    where $D \subseteq \mathbb{K}$ is the unit disk.
\end{definition}

Theorem \ref{theorem:QCBCartesianClosed} now directly implies:

\begin{lemma}\label{lemma:CbQCB}
    Let $X$ be a weakly Hausdorff QCB space. Then $C_b(X)$ is a weakly Hausdorff QCB space as well.
\end{lemma}

Another pleasant consequence of topologising $C_b(X)$ as a colimit over exponential objects is the following. 

\begin{lemma}\label{lemma:CbEnrichedFunctor}
    Let $X,Y$ be CGWH spaces. Then the pullback mapping, 
        $$ C(X,\, Y) \;\to \;C(C_b(Y),\: C_b(X)), \;\;\; f \mapsto (g \mapsto g \circ f), $$
    is continuous. In other words, $C_b$ is a $\cgwh$-enriched functor.
\end{lemma}
\begin{proof}
    It suffices to show that the uncurried map, 
        $$ C(X, Y) \times C_b(Y) \to C_b(X), \;\; (f, g) \mapsto g \circ f, $$
    is continuous. Since $C_b(Y)$ is a sequential colimit over the spaces $C(Y, nD)$ (by definition), it further suffices to verify the continuity of 
        $$ C(X,Y) \times (nD)^Y \to C_b(X), $$
    for each $n\in \mathbb{N}$. This map in turn factors through the inclusion $C(X,nD) \hookrightarrow C_b(X)$,
    reducing our claim to the continuity of
        $$ C(X,Y) \times (nD)^Y \to (nD)^X, \qquad (n \in \mathbb{N})$$
    which follows from the Cartesian closedness of $\cgwh$.
\end{proof}

Next, we give a simple description of the convergent sequences in $C_b(X)$. When $X$ is a QCB space, this completely characterises the topology of $C_b(X)$, since in this case, $C_b(X)$ is a sequential space by Lemma \ref{lemma:CbQCB}.

\begin{lemma}\label{lemma:ConvergentSequencesInCb}
    Let $X$ be a CGWH space. Then a sequence $(f_n)$ of continuous bounded functions converges to $f$ in $C_b(X)$ if, and only if, $(f_n)$ is uniformly bounded and converges in the compact-open topology to $f$.
\end{lemma}
\begin{proof}
     A convergent sequence in $C_b(X)$ can be identified with a continuous map 
        $$ \mathbb{N}\cup\{\infty\} \to C_b(X) $$
    from the one-point compactification of the natural numbers to $C_b(X)$. The image of such map being compact, \cite[Lemma 3.7]{strickland2009category} implies that every convergent sequence lies in one of the spaces $(nD)^X$ (for some $n\in \mathbb{N}$) and is hence uniformly bounded. Moreover, a sequence converges in $(nD)^X$ if, and only if, it converges in the compact-open topology (see Lemma \ref{lemma:ConvergenceInYtoTheX}).
\end{proof}

The following lemma will be needed later, to show that on a compact Hausdorff space, every Baire measure is $k$-regular; see Lemma \ref{lemma:radonMeasuresKregular}.

\begin{lemma}\label{lemma:continuityOnCb}
    Let $X$ be a CGWH space and let $\phi$ be a not-necessarily-continuous functional on $C_b(X)$. Assume that for every uniformly bounded net of continuous functions $(f_i)$ on $X$ which converges to $0$ in the compact-open topology, $\phi(f_i) \to 0$. Then $\phi$ is continuous.
\end{lemma}
\begin{proof}
    By assumption, $\phi$ is continuous on each space of continuous functions $C_{c.o.}(X, nD)$ with the compact-open topology, implying that it is also continuous on each $(nD)^X=kC_{c.o.}(X, nD)$. The claim follows, since $C_b(X)$ is given as the colimit over these spaces. 
\end{proof}

\subsection{The natural dual of a linear CGWH space} The Cartesian closedness of $\cgwh$ suggests a canonical topology on the continuous dual space of a linear CGWH space. When equipped with this topology, we call it the \emph{natural dual}.

\begin{definition}\label{definition:naturalDual}
    Let $V$ be a linear CGWH space. The \emph{natural dual} of $V$, 
        $$ V^\wedge := \{f \in C(V) \mid f \:\mathrm{linear}\} \subseteq C(V), $$
    is the space of continuous linear functionals on $V$, topologised as a closed subspace of the space of continuous maps $C(V)$.
\end{definition}

As a direct consequence of Lemma \ref{lemma:closedSubspacesQCB}, we obtain:

\begin{lemma}\label{lemma:naturalDualQCB}
    If a linear CGWH space $V$ is a QCB space, then so is its natural dual $V^\wedge$.
\end{lemma}

The next lemma will be needed to compare the topology on the space $\mc{M}(X)$ of $k$-regular measures (see Section \ref{sec:kRegularMeasures}) to the topology of weak convergence; see Theorem \ref{convergence_of_measures_on_compactum}.

\begin{lemma}\label{lemma:unitBallOfNaturalDualCarriesWeakStarTop}
    Let $V$ be a Banach space, considered as a linear CGWH space. Then, on the unit ball of $V^\wedge$, the topology of $V^\wedge$ coincides with the weak-$*$ topology.
\end{lemma}
\begin{proof}
    By a corollary of the uniform boundedness principle (see \cite[p.~86, 4.6]{schaefer1971topological}), the weak-$*$ topology on $V'$ coincides with the compact-open topology on the unit ball of $V'$. Since the unit ball is compact in the weak-$*$ topology, by the Banach-Alaoglu theorem, it is hence also compact in the compact-open topology. On compact subsets, the compact-open topology agrees with its $k$-ification, which is precisely the topology of $V^\wedge$. 
\end{proof}

\section{A Riesz representation theorem}

In this section, we show that the natural dual $C_b(X)^\wedge$ of the space of continuous bounded functions on a CGWH space $X$ can be identified with a certain class of measures, the \emph{$k$-regular Baire measures} on $X$. When $X$ is a QCB space, every Baire measure is $k$-regular and we obtain a particularly natural version of the Riesz representation theorem: the natural dual of $C_b(X)$ can be identified with the space of Baire measures on $X$ (see Corollary \ref{corollary:RieszForQCB}).

\begin{warning}
    In the following, the term ``measure'' will be reserved for countably additive $\mathbb{K}$-valued measures of bounded variation. We will not consider infinite measures and do not necessarily assume measures to be positive. In general, we will follow the measure-theoretic terminology of \cite{bogachev2007measure}. 
\end{warning}

\subsection{Baire, Borel and Radon measures}

Let $X$ be a topological space. We recall the definitions of Baire, Borel and Radon measures (which unfortunately are not quite uniform throughout the literature).

\begin{definition}\label{def:Baire}
    The \emph{Baire $\sigma$-algebra} $\mc{B}a(X)$ on $X$ is the $\sigma$-algebra generated by all continuous $\R$-valued functions on $X$. A \emph{Baire measure} is a $\mathbb{K}$-valued measure on the Baire $\sigma$-algebra. \par 
    A \emph{Borel measure} on $X$ is a measure on the Borel $\sigma$-algebra $\mc{B}(X)$ (which is the $\sigma$-algebra generated by all open sets). A \emph{Radon measure} on $X$ is a Borel measure $\mu$ on $X$ such that for each $A\in \mc{B}(X)$ and $\epsilon >0$, there is a compact subset $K \subseteq A$ such that $|\mu|(A \setminus K) < \epsilon$.
\end{definition}

When $X$ is sufficiently well-behaved, the notions of Baire and Radon measure essentially coincide:

\begin{theorem}\label{theorem:baireMeasuresVsRadonMeasures}
    On a Polish space, the Baire and Borel $\sigma$-algebras agree; moreover the notions of Baire, Borel, and Radon measure coincide.
    On a compact Hausdorff space, every Baire measure admits a unique extension to a Radon measure.
\end{theorem}
\begin{proof}
    For the statement about Polish spaces, see \cite[p. 70, Theorem 7.1.7]{bogachev2007measure}, for the one on compacta, see \cite[p. 81, Theorem 7.3.4]{bogachev2007measure}.
\end{proof}

\subsection{\texorpdfstring{$k$}{k}-regular measures}\label{sec:kRegularMeasures} Let $X$ be a CGWH space.

\begin{definition}\label{definition:kRegularMeasures}
    A \emph{$k$-regular measure} on $X$ is a Baire measure $\mu$ on $X$ for which the integration map,
        $$ C_b(X) \to \K, \;\;  f \mapsto \int_X f \:\mathrm{d} \mu, $$
    is continuous, where $C_b(X)$ carries its natural CGWH topology (see Definition \ref{definition:CbCGWHSpace}). We denote the set of $k$-regular measures on $X$ by $\mathcal{M}(X)$.
\end{definition}

The condition of $k$-regularity is a very weak regularity condition to impose on Baire measures, as the following lemma shows. 

\begin{lemma}\label{lemma:radonMeasuresKregular}
    Let $X$ be either a compact Hausdorff space or a weakly Hausdorff QCB space.
    Then every Baire measure is a $k$-regular measure.
\end{lemma}
\begin{proof}
    We consider both cases separately.\par 
    1. Let $X$ be a compact Hausdorff space and let $\mu$  be a Baire measure on $X$. Let $(f_i)$ be a net of continuous functions on $X$ uniformly bounded by $C>0$ converging uniformly on compact subsets to $0$. By Lemma \ref{lemma:continuityOnCb}, it suffices to show that 
        $$ \int f_i \:\mathrm{d} \mu \to 0. $$
    Let $\epsilon >0$. By Theorem \ref{theorem:baireMeasuresVsRadonMeasures}, $\mu$ has a unique extension to a Radon measure $\tilde{\mu}$ on $X$. Let $\epsilon >0$. Then, since $\tilde{\mu}$ is a Radon measure, there exists a compact subset $K$ such that $|\tilde{\mu}|(X\mathbin{\backslash} K)<\epsilon$. Now, by the compact convergence of $(f_i)$, there is an index $i_0$ such that for all $i\geq i_0$, 
        $$ \sup_{x\in K} |f_i(x)| < \epsilon. $$
    Now, for all $i\geq i_0$, 
    \begin{align*}
        \Big| \int f_i \:\mathrm{d} \mu \Big| \leq \int |f_i| \:\mathrm{d} |\tilde{\mu}| &= 
        \int_K |f_i| \:\mathrm{d} |\tilde{\mu}| + \int_{X\mathbin{\backslash} K} |f_i| \:\mathrm{d} |\tilde{\mu}| \\
        &\leq \sup_{x\in K} |f_i(x)| |\tilde{\mu}|(K) + C |\tilde{\mu}|(X \mathbin{\backslash} K) \\
        &< (|\mu|(X) + C) \,\epsilon.
    \end{align*}
    This shows convergence and, in conclusion, that $\mu$ is $k$-regular.\par
    2. Let $X$ be a weakly Hausdorff QCB space and let $\mu$ be a Baire measure on $X$. We want to show that 
        $$ I_\mu: C_b(X) \to \K, \;\;  f \mapsto \int_X f \:\mathrm{d} \mu, $$
    is continuous. Since $X$ is a QCB space, so is $C_b(X)$ (by Lemma \ref{lemma:CbQCB}). Hence, $C_b(X)$ is a sequential space (by Proposition \ref{proposition:propertiesOfQCBspaces}) and it suffices to show that $I_\mu$ is sequentially continuous. So let $f_n \to f$ be a convergent sequence in $C_b(X)$. By Lemma \ref{lemma:ConvergentSequencesInCb}, $(f_n-f)$ is uniformly bounded. Therefore, by the dominated convergence theorem,
        $$ |I_\mu(f_n)-I_\mu(f)|\leq \int_X |f_n - f| \:\mathrm{d} |\mu| \to 0, $$
    showing that $I_\mu$ is indeed sequentially continuous. 
\end{proof}

\begin{representationtheorem}\label{theorem:RieszForCbAndM}
    The mapping, 
        $$\mathcal{M}(X) \to C_b(X)^\wedge, \;\; \mu \mapsto \int_X - \:\mathrm{d} \mu, $$
    is a bijection between the set of $k$-regular measures on $X$ and the set of continuous linear functionals on $C_b(X)$.
\end{representationtheorem}
\begin{proof}
    Assume first that $X$ is Hausdorff. We need to show that for every $\phi \in C_b(X)^\wedge$, there exists a unique Baire measure $\mu$ such that 
        $$ \phi(f) = \int f \:\mathrm{d} \mu, $$
    for all $f\in C_b(X)$. By a general version of the Riesz representation theorem for Hausdorff spaces \cite[p.~111, Theorem 7.10.1]{bogachev2007measure}, 
    this is the case if, and only if, for every monotonically decreasing sequence $(f_n)$ in $C_b(X)$ converging pointwise to zero, 
    we have $\phi(f_n)\to 0$. 
    So let $(f_n)$ be such a sequence. 
    By Dini's theorem, $(f_n)$ converges to $0$ uniformly on compact subsets of $X$. 
    Moreover, since it is monotonically decreasing, $(f_n)$ is uniformly bounded by the constant $\sup_{x\in X} f_0(x)$. 
    Hence, by Lemma \ref{lemma:ConvergentSequencesInCb}, 
    $(f_n)$ converges to $0$ in $C_b(X)$ and by continuity of $\phi$, $\phi(f_n) \to 0$, which completes the proof of the  Hausdorff case. \\
    We now turn to the general case, assuming only that $X$ is a \emph{weakly} Hausdorff CG space. This will follow from applying the Hausdorff case to a certain Hausdorff space $\tilde{X}$ with the property that $C_b(X)=C_b(\tilde{X})$. \\
    Recall that a topological space $Y$ is \emph{functionally Hausdorff} if for any two points $y_1,y_2 \in Y$ with $f(y_1)=f(y_2)$ for all $f\in C(Y)$, we have $y_1=y_2$. 
    Functionally Hausdorff CG spaces form a reflective subcategory of $\mathsf{CGWH}$. 
    The functionally Hausdorff reflection $\tilde{X}$ of $X$ is given by the quotient of $X$ by the equivalence relation that identifies two points if the values of any real-valued function at both points coincide.
    Moreover, functionally Hausdorff CG spaces form an exponential ideal in $\mathsf{CGWH}$, i.e.~for any CG space $X$ and any functionally Hausdorff space $Y$, we have that $Y^X$ is functionally Hausdorff. 
    This implies that $Y^{\tilde{X}} = Y^X$, or any CG space $X$ and any functionally Hausdorff space $Y$, and therefore we also have $C(X)=C(\tilde{X})$ and $C_b(X)=C_b(\tilde{X})$, both as sets and as topological spaces. 
    Hence, a $k$-regular measure on $X$ is equivalently a $k$-regular measure on $\tilde{X}$ and a continuous linear functional on $C_b(X)$ is equivalently a continuous linear functional on $C_b(\tilde{X})$. Since any functionally Hausdorff space is in particular a Hausdorff space, the general case now follows from the Hausdorff case applied to $\tilde{X}$, as desired.
\end{proof}

As a direct consequence of Lemma \ref{lemma:radonMeasuresKregular} and Theorem \ref{theorem:RieszForCbAndM}, we obtain:

\begin{corollary}\label{corollary:RieszForQCB}
    For $X$ a weakly Hausdorff QCB space, the natural dual of $C_b(X)$ can be identified with the space of Baire measures on $X$.
\end{corollary}

We will henceforth identify the set of $k$-regular measures $\mc{M}(X)$ with $C_b(X)^\wedge$. In particular, we topologise $\mc{M}(X)$ according to the topology on $C_b(X)^\wedge$. In this way, we obtain an endofunctor, 
    $$ \mc{M}: \cgwh \to \cgwh, \;\; X \mapsto \mc{M}(X), \; f \mapsto f_*. $$
The next step will be to endow this endofunctor $\mc{M}$ with the structure of a monad. 

\section{The Riesz and Baire monads}\label{sec:TheMonad}

\subsection{The Riesz monad}

We now describe the monad structure of the \emph{Riesz monad} $\mathcal{M}$. Unit and multiplication are given as follows.

\begin{definition}\label{definition:UnitAndMultiplicationOfM}
    Define the (families of) map(s),
        $$ \delta_\bullet : X \mapsto \mathcal{M}(X), \;\; x \mapsto \delta_x, \;\;\;\;\; (X \in\cgwh) $$
    and 
        $$ \Rbag: \mathcal{M}(\mathcal{M}(X)) \to \mathcal{M}(X),\;\; \Rbag(\pi)(f) := \int_{\mathcal{M}(X)}\mu(f)\:\mathrm{d} \pi(\mu). \;\;\;\;\; (X \in\cgwh) $$
\end{definition}

Note that under the identification of measures with functionals via Theorem \ref{theorem:RieszForCbAndM}, we have that 

\begin{equation}\label{equation:monadMultiplicationViaFunctionals}
    \Rbag (\pi)(f) = \pi( \mu \mapsto \mu(f) ), \qquad (\pi \in \mathcal{M}(\mathcal{M}(X)), \: f \in C_b(X), \: X \in \cgwh)
\end{equation}

as well as, 

\begin{equation}\label{equation:pushforwardViaFunctionals}
    f_* \mu = \mu( - \circ f ). \qquad (f: X \to Y, \: \mu \in \mc{M}(X), \: X,Y \in \cgwh)
\end{equation}

Eq.~\eqref{equation:monadMultiplicationViaFunctionals} and \eqref{equation:pushforwardViaFunctionals} display $\mathcal{M}$ as a submonad of the double dualisation monad $\mathbb{K}^{\mathbb{K}^{(-)}}$, which is also know as the \emph{continuation monad} in the context of programming language theory. \\
As a first step towards showing that $\delta_\bullet$ and $\Rbag$ equip $\mc{M}$ with the structure of a monad, we will use this to verify:

\begin{lemma}
    The families of maps $\delta_\bullet$ and $\Rbag$ constitute natural transformations.
\end{lemma}
\begin{proof}
    We need to show $\delta_\bullet$ and $\Rbag$ are continuous and make the necessary naturality square commute.
    
    1. \emph{The map $\delta_\bullet$ is continuous:} Under the identification of measures and functionals, we need show that the map 
        $$ X \to C_b(X)^\wedge, \;\; x \mapsto \delta_\bullet $$
    is continuous. This map factors as the composite, 
    % https://q.uiver.app/#q=WzAsMyxbMCwwLCJYIl0sWzEsMCwiQyhYKV5cXHdlZGdlIl0sWzMsMCwiQ19iKFgpXlxcd2VkZ2UiXSxbMCwxLCJcXGRlbHRhX1xcYnVsbGV0Il0sWzEsMiwiKFxcY2RvdCl8X3tDX2IoWCl9Il1d
    \[\begin{tikzcd}
	   X & {C(X)^\wedge} && {C_b(X)^\wedge}
	   \arrow["{\delta_\bullet}", from=1-1, to=1-2]
	   \arrow["{(-)|_{C_b(X)}}", from=1-2, to=1-4]
    \end{tikzcd}\]
    so its continuity follows from Cartesian closedness of $\cgwh$ and the fact that the inclusion map $C_b(X) \to C(X)$ is continuous. 
        
    2. \emph{Naturality of $\delta_\bullet$:} We need to verify that for all CGWH spaces $X,Y$ and any continuous map $f:X\to Y$ the following diagram commutes:
    % https://q.uiver.app/#q=WzAsNCxbMCwwLCJYIl0sWzEsMCwiWSJdLFswLDEsIlxcbWF0aGNhbHtNfShYKSJdLFsxLDEsIlxcbWF0aGNhbHtNfShZKSJdLFswLDEsImYiXSxbMCwyLCJcXGRlbHRhX1xcYnVsbGV0IiwyXSxbMSwzLCJcXGRlbHRhX1xcYnVsbGV0Il0sWzIsMywiZl8qIiwyXV0=
    \[\begin{tikzcd}
	   X & Y \\
	   {\mathcal{M}(X)} & {\mathcal{M}(Y)}
	   \arrow["f", from=1-1, to=1-2]
	   \arrow["{\delta_\bullet}"', from=1-1, to=2-1]
	   \arrow["{\delta_\bullet}", from=1-2, to=2-2]
	   \arrow["{f_*}"', from=2-1, to=2-2]
    \end{tikzcd}\]
    This is to say that for all $x\in X$,
        $$ f_*\delta_x = \delta_{f(x)}, $$
    which is true.
    
    3. \emph{The map $\Rbag$ is continuous:} Being given by a certain evaluation-type map (see Eq. \eqref{equation:monadMultiplicationViaFunctionals}), the continuity of $\Rbag$ follows from Cartesian closedness of $\cgwh$.
        
    4. \emph{Naturality of $\Rbag$:} We want to show that the following diagram commutes:
    % https://q.uiver.app/#q=WzAsNCxbMCwwLCJcXG1hdGhjYWx7TX0oXFxtYXRoY2Fse019KFgpKSJdLFsxLDAsIlxcbWF0aGNhbHtNfShcXG1hdGhjYWx7TX0oWSkpIl0sWzAsMSwiXFxtYXRoY2Fse019KFgpIl0sWzEsMSwiXFxtYXRoY2Fse019KFkpIl0sWzAsMSwiKGZfKilfKiJdLFswLDIsIlxcc21hbGxpbnR7fSIsMl0sWzEsMywiXFxzbWFsbGludHt9Il0sWzIsMywiZl8qIiwyXV0=
    \[\begin{tikzcd}
	    {\mathcal{M}(\mathcal{M}(X))} & {\mathcal{M}(\mathcal{M}(Y))} \\
	   {\mathcal{M}(X)} & {\mathcal{M}(Y)}
	    \arrow["{(f_*)_*}", from=1-1, to=1-2]
	    \arrow["{\Rbag}"', from=1-1, to=2-1]
	    \arrow["{\Rbag}", from=1-2, to=2-2]
	    \arrow["{f_*}"', from=2-1, to=2-2]
    \end{tikzcd}\]
    To show this, we calculate that, for all $\pi\in \mathcal{M}(\mathcal{M}(X))$ and all $h\in C_b(X)$,
    \begin{align*}
            f_*\left(\Rbag(\pi)\right)(h) 
            &= \:\Rbag(\pi)(h \circ f)\tag{Eq.~\eqref{equation:pushforwardViaFunctionals}}  \\
            &= \pi(\mu \mapsto \mu(h \circ f)) \tag{Eq.~\eqref{equation:monadMultiplicationViaFunctionals}} \\
            &= \pi(\mu \mapsto f_*\mu(h)) 
            \tag{Eq.~\eqref{equation:pushforwardViaFunctionals}} \\
            &= \pi(f_* \circ(\mu \mapsto \mu(h)))\\
            &= (f_*)_*\pi(\mu \mapsto \mu(h))
            \tag{Eq.~\eqref{equation:pushforwardViaFunctionals}}\\
            &= \:\Rbag((f_*)_*\pi)(h), \tag{Eq.~\eqref{equation:monadMultiplicationViaFunctionals}} 
    \end{align*}
    which completes the proof.
\end{proof}

\begin{theorem}
    $(\mathcal{M}, \delta_\bullet, \Rbag)$ is a monad on $\cgwh$, the \emph{Riesz monad}.
\end{theorem}
\begin{proof}
    We need to verify the three monad laws. Let $X$ is a CGWH space, $\mu \in \mathcal{M}(X)$, $f\in C_b(X)$ and $\pi \in \mathcal{M}(\mathcal{M}(X))$.

    1. \emph{First unit law:}
        $$\Rbag(\delta_\mu)(f) = \int \nu(f) \:\mathrm{d} \delta_\mu(\nu) = \mu(f). $$ 
        
    2. \emph{Second unit law:}
    $$(\Rbag (\delta_\bullet)_*\mu)(f) = \int \delta_x(f) \:\mathrm{d} \mu = \int f \:\mathrm{d} \mu = \mu(f). $$
    
    3. \emph{Associativity law:} We need to show that for all $\Pi\in \mathcal{M}(\mathcal{M}(\mathcal{M}(X))),h \in C_b(X)$,
        $$\Rbag \left(\Rbag (\Pi)\right)(h)= \Rbag \left(\Rbag\right)_*\Pi. $$
    We calculate, 
    \begin{align*}
        \Rbag \left(\Rbag (\Pi)\right)(h) 
        &= \left(\Rbag(\Pi)\right)( \mu \mapsto \mu(h) ) \tag{Eq.~\eqref{equation:monadMultiplicationViaFunctionals}} \\
        &= \Pi(\nu \mapsto \nu(\mu \mapsto \mu(h))) \tag{Eq.~\eqref{equation:monadMultiplicationViaFunctionals}} \\
        &= \Pi(\nu \mapsto \left(\Rbag (\nu)\right)(h)) \tag{Eq.~\eqref{equation:monadMultiplicationViaFunctionals}} \\
        &= \Pi\left([\mu \mapsto \mu(h)]\circ \Rbag\right) \\
        &= \left(\Rbag\right)_*\Pi (\mu \mapsto \mu(h)) \tag{Eq.~\eqref{equation:pushforwardViaFunctionals}}\\
        &= \:\Rbag (\left(\Rbag\right)_*\Pi) \tag{Eq.~\eqref{equation:monadMultiplicationViaFunctionals}},
    \end{align*}
    completing the proof.
\end{proof}

\subsection{The Baire monad}

By Lemmas \ref{lemma:naturalDualQCB} and \ref{lemma:CbQCB}, $\mathcal{M}(X)$ is a QCB space whenever $X$ is a (weakly Hausdorff) QCB space, and in this case, $\mathcal{M}(X)$ consists exactly of the Baire measures on $X$. The Riesz monad therefore restricts to a monad on $\mathsf{QCB}_{h}$, which we refer to as the \emph{Baire monad}.

\subsection{Enriched structure: a strengthened continuous mapping theorem}\label{sec:EnrichedStructure}

A basic fact concerning convergence of random variables in distribution is the \emph{continuous mapping theorem} \cite{mann1943stochastic}, one version of which can be formulated as follows. Let $f:X\to Y$ be a continuous map between metric spaces $X,Y$, and let $(x_n)$ be a sequence of $X$-valued random variables converging in distribution to an $X$-valued random variable $x$. Then $f(x_n)$ converges in distribution to $f(x)$. In other words, any weakly convergent sequence of Borel probability measures $(\mu_n)$ is mapped to a weakly convergent sequence $(f_*\mu_n)$ under the pushforward operation $f_*$. When $X,Y$ are Polish spaces we can formulate this even more concisely as the continuity of the pushforward
    $$ f_*: \mc{P}(X) \to \mc{P}(Y), $$
between spaces of probability measures with the topology of weak convergence. \par 
Cartesian closedness enables a stronger statement in our setting, giving $\mc{M}(X)$ the structure of an \emph{enriched monad}. 

\begin{proposition}\label{proposition:enrichedContinuityM}
    Let $X, Y$ be CGWH spaces. Then the map 
        $$ C(X,Y) \:\to \:C(\mc{M}(X), \,\mc{M}(Y)), \;\;\; f \mapsto f_*$$
    is continuous.
\end{proposition}
\begin{proof}
    This map factors as the composite,
    % https://q.uiver.app/#q=WzAsMyxbMCwwLCJDKFgsIFkpICJdLFsxLDAsIkMoQ19iKFkpLCBDX2IoWCkpIl0sWzIsMCwiQyhDX2IoWCleXFx3ZWRnZSwgQ19iKFkpXlxcd2VkZ2UpPUMoXFxtYXRoY2Fse019KFgpLCBcXG1hdGhjYWx7TX0oWSkpIl0sWzAsMSwiKC0pXioiXSxbMSwyLCIoLSleXFx3ZWRnZSJdXQ==
    \[\begin{tikzcd}
	{C(X, Y) } & {C(C_b(Y), C_b(X))} & {C(C_b(X)^\wedge, C_b(Y)^\wedge)=C(\mathcal{M}(X), \mathcal{M}(Y))}.
	\arrow["{(-)^*}", from=1-1, to=1-2]
	\arrow["{(-)^\wedge}", from=1-2, to=1-3]
    \end{tikzcd}\]
    The first of these maps is continuous by Lemma \ref{lemma:CbEnrichedFunctor}, while the continuity of the second follows from the definition of the natural dual $(-)^\wedge$ and the Cartesian closedness of $\cgwh$.
\end{proof}

As we will show in Section \ref{sec:relationToGiryMonad}, when $X$ is a Polish space, the space of probability measures $\mc{P}(X)$ with the weak topology is a closed subspace of of $\mathcal{M}(X)$. Hence, Proposition \ref{sec:relationToGiryMonad} can indeed be interpreted as a strengthened continuous mapping theorem, reducing to the following statement in the case of Polish spaces. 

\begin{example}
    Let $X, Y$ be complete separable metric spaces and let $(f_n)$ be a sequence of continuous maps $X\to Y$ converging uniformly on compact sets to $f: X \to Y$. Moreover, let $(x_n)$ be a sequence of random variables on $X$ converging in distribution to a random variable $x$ on $X$, i.e.~the sequence $(\mu_n)$ of their distributions converging weakly to the distribution $\mu \in \mc{P}(X)$ of $x$. Then $f_n(x_n)$ converges to $f(x)$ in distribution, i.e.~the sequence of pushforwards $((f_n)_*\mu_n)$ converges weakly to $f_*\mu$.
\end{example}

Note that, in contrast to the classical continuous mapping theorem, this is a statement about simultaneous convergence in \emph{both $(f_n)$ and $(x_n)$}.  \\
Finally, we remark that Proposition \ref{proposition:enrichedContinuityM} implies that $\mc{M}$ is an \emph{enriched monad}. Recall that an $\mathsf{V}$-enriched monad on a $\mathsf{V}$-enriched category $\mathsf{C}$ is a $\mathsf{V}$-enriched endofunctor $T$ on $\mathsf{C}$ equipped with two $\mathsf{V}$-enriched natural transformations making the usual diagrams commute. Cartesian closed categories are canonically enriched over themselves, and Proposition \ref{proposition:enrichedContinuityM} now implies:

\begin{corollary}
    The Riesz monad $\mc{M}$ is canonically a $\cgwh$-enriched monad on $\cgwh$. Likewise, the Baire monad $\mc{M}$ is canonically an $\mathsf{QCB}_{h}$-enriched monad on $\mathsf{QCB}_{h}$.
\end{corollary}

\subsection{Commutativity: product measures and Fubini's theorem}\label{sec:commutativityAndFubini}

Any enriched monad is canonically a strong monad \cite[Section 3]{ratkovic2013morita}. One may therefore ask whether $\mc{M}$ is a \emph{commutative} strong monad, or equivalently \cite{kock1972strong}, a symmetric monoidal monad. This does indeed hold in the case of the Baire monad (Theorem \ref{M_bbd_meas_monad_commutative} below), and this fact is closely related to product measures and Fubini's theorem. \par 
The first crucial observation in this direction concerns the Baire $\sigma$-algebra on a product of QCB spaces.

\begin{lemma}\label{lemma:productMeasuresOnQCBspaces}
    Let $X, Y$ be QCB spaces. Then,
        $$ \mc{B} a (X\times Y) = \mc{B} a (X) \otimes \mc{B} a (Y), $$
    where 
        $$  \mc{B} a (X) \otimes \mc{B} a (Y) =  \sigma(\{A \times B \,|\, A \in \mc{B} a (X), B \in \mc{B} a (Y)\}) $$
    is the product $\sigma$-algebra of the Baire $\sigma$-algebras on each factor.
\end{lemma}
\begin{proof}
    For the inclusion 
        $$ \mc{B} a (X) \otimes \mc{B} a (Y) \subseteq \mc{B} a (X\times Y), $$
    let $A = f^{-1}(0)$ and $B= g^{-1}(0)$ with $f\in C(X), g\in C(Y)$. Since $\mc{B} a (X) \otimes \mc{B} a (Y)$ is generated by cartesian products of the form $A\times B$, it suffices that $A\times B \in \mc{B} a (X\times Y)$, which holds, as $A\times B = (f\cdot g)^{-1}(0)$. \\
    For the reverse inclusion, it suffices to show that every continuous map $X\times Y \to \mathbb{R}$ is measurable with respect to the product $\sigma$-algebra $\mc{B} a (X) \otimes \mc{B} a (Y)$. We will show the stronger claim that, in fact, every \emph{separately} continuous function $X\times Y \to \mathbb{R}$ is $\mc{B} a (X) \otimes \mc{B} a (Y)$-measurable. In order to achieve this, we will apply certain point-set topological results regarding so-called \emph{Rudin spaces}. \\
    A topological space $Z$ is said to be a \emph{Rudin space} if every separately continuous function on $Z \times W$, where $W$ is any other topological space, is the pointwise limit of a sequence of continuous functions on $Z \times_{\mathsf{Top}} W$. We will show that every QCB space is a Rudin space. This implies the original claim, since every continuous function on $X\times_{\mathsf{Top}} Y$ is $\mc{B} a (X) \otimes \mc{B} a(Y)$-measurable (by \cite[Proposition 6.10.7]{bogachev2007measure} together with the fact that QCB spaces are hereditarily Lindelöf, see Proposition \ref{proposition:propertiesOfQCBspaces}) and pointwise limits of measurable functions are measurable.  \\ 
    Let $Z$ be any QCB space. By \cite[Theorem 6.2.(2)]{banakh2008metrically}, $Z$ is a Rudin space if $C_p(Z)$ is a Rudin space, where $C_p(Z)$ is the space of continuous functions on $Z$ with the pointwise topology (i.e.~the topology that it inherits from the $\mathsf{Top}$-product of $|Z|$-fold many copies of the real line). Now, $C_p(Z)$ has a countable network, since it is the continuous image of the QCB space $C(Z)$ under the identity map. Moreover, as a topological vector space, $C_p(Z)$ is normal. Hence, by the Lindelöf property, $C_p(Z)$ is paracompact. As a paracompact space with a countable (in particular, $\sigma$-discrete) network, by \cite[Theorem 6.2.(3), 2.3.(1)]{banakh2008metrically}, $C_p(Z)$ is Rudin space and hence, so is $Z$. This shows that every QCB space is a Rudin space, which completes the proof.
\end{proof}

As a consequence, the measure-theoretic product measure $\mu\otimes \nu$ is a Baire measure on the product $X\times Y$, for any Baire measures $\mu$ and $\nu$ on QCB spaces $X,Y$. Building on this observation, we show that the formation of product measures constitutes a continuous map.

\begin{lemma}
    Let $X, Y$ be weakly Hausdorff QCB spaces. Then the map, 
        $$ \otimes: \mathcal{M}(X) \times \mathcal{M}(Y) \to \mathcal{M}(X\times Y), \;\; (\mu, \nu) \mapsto \mu\otimes \nu, $$
     is continuous.
\end{lemma}
\begin{proof}
    Identifying Baire measures with functionals, $\otimes$ is given by,
        $$ (\mu \otimes \nu) (f) = \nu(y \mapsto \mu(x \mapsto f(x,y))), \quad (\mu \in  \mc{M}(X), \nu \in  \mc{M}(Y), f \in C_b(X\times Y)), $$
    whereby the continuity of $(\mu, \nu) \mapsto \mu \otimes \nu$ follows from the cartesian closedness of $\mathsf{QCB}_h$.
\end{proof}

% Observation: formation of product measures is continuous in our setting, but it not if one uses the weak topology (see https://mathoverflow.net/questions/211109/is-taking-the-product-of-signed-measures-weakly-continuous). Tentative conclusion: when working with signed measures, the weak topology is ``wrong'' and one should use the one on $\mc{M}$ instead.

Now, the family of maps  
    $$ \otimes: \mathcal{M}(X) \times \mathcal{M}(Y) \to \mathcal{M}(X\times Y), \;\; (\mu, \nu) \mapsto \mu\otimes \nu, \;\;\;\; ((X,Y)\in \mathsf{QCB}_h \times \mathsf{QCB}_h)$$
is a natural transformation and we have that:

\begin{theorem}\label{M_bbd_meas_monad_commutative}
    $(\mathcal{M}, \delta_\bullet, \Rbag, \otimes)$ is a symmetric monoidal monad on $\mathsf{QCB}_h$. Equivalently \cite{kock1972strong}, $\mathcal{M}$ is a commutative strong monad with respect to the associated left strength, 
        $$ \lambda^{X,Y}: X \times \mathcal{M}(Y) \to \mathcal{M}(X \times Y), \;\; (x, \nu) \mapsto \delta_x \otimes \nu, $$
    and right strength, 
        $$ \rho^{X,Y}: \mathcal{M}(X) \times Y \to \mathcal{M}(X \times Y), \;\; (\mu, y) \mapsto \mu \otimes \delta_y. $$
\end{theorem}
\begin{proof}
    We need to show that the following diagram commutes:
    \begin{center}
    \begin{tikzcd}[%
    nodes={scale=0.9}, arrows={thick},%
    ]
    & \mathcal{M}(X) \times \mathcal{M}(Y) \ar{dl}[swap]{\lambda} \ar{dr}{\rho}  \\
    \mathcal{M}(\mathcal{M}(X) \times Y) \ar{d}[swap]{\mathcal{M}(\rho)} && \mathcal{M}(X\times \mathcal{M}(Y)) \ar{d}{\mathcal{M}(\lambda)} \\
    \mathcal{M}(\mathcal{M}(X\times Y)) \ar{dr}[swap]{\Rbag} && \mathcal{M}(\mathcal{M}(X\times Y)) \ar{dl}{\Rbag} \\
    & \mathcal{M}(X\times Y)
    \end{tikzcd}
    \end{center}
    Let $\mu_0\in \mathcal{M}(x), \nu_0\in\mathcal{M}(Y)$ and $f\in C_b(X\times Y)$. Then, 
    \begin{align*}
        (\Rbag \circ \mathcal{M}(\rho) \circ \lambda (\mu_0, \nu_0))(f) 
        &= \;[\Rbag (\mathcal{M}(\rho)(\delta_{\mu_0} \otimes \nu_0) )](f) \\
        &= \;[\Rbag (F \mapsto (\delta_{\mu_0} \otimes \nu_0)[(\mu, y) \mapsto F(\mu \otimes \delta_y)])](f) \tag{Eq.~\eqref{equation:pushforwardViaFunctionals}}\\
        &= \;[\Rbag (F \mapsto \nu_0(y \mapsto F(\mu_0 \otimes \delta_y)])](f) \\
        &= \nu_0(y \mapsto (\mu_0 \otimes \delta_y)(f)) \tag{Eq.~\eqref{equation:monadMultiplicationViaFunctionals}}\\
        &= \nu_0(y \mapsto \mu_0(x \mapsto f(x, y)))\\
        &= \int_Y \int_X f(x, y) \mathrm{d} \mu_0(x) \mathrm{d} \nu_0(y).
    \end{align*}
    By symmetry, we also have,
    $$ (\Rbag \circ \mathcal{M}(\lambda) \circ \rho (\mu_0, \nu_0))(f) = \int_X \int_Y f(x, y) \mathrm{d} \nu_0(y) \mathrm{d} \mu_0(x).$$
    Hence, our claim reduces to,
        \begin{equation}\label{fubini_for_Cb}
            \int \int f(x,y) \:\mathrm{d} \mu(x) \:\mathrm{d} \nu(y)  = \int \int f(x,y) \:\mathrm{d} \nu(y) \:\mathrm{d} \mu(x).
        \end{equation}
    which follows from Fubini's theorem, using Lemma \ref{lemma:productMeasuresOnQCBspaces}.
\end{proof}

\section{The Riesz and Baire probability monads}

\subsection{Passing to probability measures}\label{sec:passingToPX}

Up to this point, our discussion has focused on $\mathbb{K}$-valued measures. We now show how to apply this to the construction of two \emph{probability} monads, the \emph{Riesz probability monad} on $\cgwh$ and the \emph{Baire probability monad} on $\mathsf{QCB}_{h}$. 

\begin{definition}
    Let $X$ be a CGWH space. We define the CGWH space of $k$-regular probability measures as the closed subspace,
        $$ \mc{P}(X) := \{\mu \in \mc{M}(X) \,|\, \mu \;\mathrm{ probability} \;\mathrm{measure}\} \subseteq \mc{M}(X), $$
    of $\mc{M}(X)$. 
\end{definition}

Note that $\mc{P}(X)$ is a QCB space whenever $X$ is a (weak Hausdorff) QCB space (by Lemma \ref{lemma:closedSubspacesQCB}), in which case $\mc{P}(X)$ consists exactly of the Baire probability measures on $X$. Since the pushforward of a probability measure is a probability measure, we obtain an endofunctor $\mc{P}$ on $\cgwh$ and on $\mathsf{QCB}_h$. We equip $\mc{P}$ with the structure of a (symmetric monoidal, in the case of $\mathsf{QCB}_h$) monad in the same way as for $\mc{M}$. The unit of $\mc{P}$ is therefore given by,
    $$ \delta_\bullet : X \mapsto \mathcal{P}(X), \;\; x \mapsto \delta_x, \;\;\;\;\; (X \in\cgwh) $$
the multiplication by,
    $$ \Rbag: \mathcal{P}(\mathcal{P}(X)) \to \mathcal{P}(X),\;\; \Rbag(\pi)(f) := \int_{\mathcal{P}(X)}\mu(f)\:\mathrm{d} \pi(\mu). \;\;\;\;\; (X \in\cgwh) $$
and, in the case of the Baire probability monad, the symmetric monoidal structure on $\mc{P}$ is given by,
    $$ \otimes: \mathcal{P}(X) \times \mathcal{P}(Y) \to \mathcal{P}(X\times Y), \;\; (\mu, \nu) \mapsto \mu\otimes \nu, \;\;\;\; ((X,Y)\in \mathsf{QCB}_h \times \mathsf{QCB}_h)$$
In conclusion, we obtain, exactly as for $\mc{M}$:

\begin{theorem}\label{theorem:rieszProbabilityMonad}
    $(\mc{P}, \delta_\bullet, \Rbag)$ is a monad on $\cgwh$, the \emph{Riesz probability monad}, which restricts to a symmetric monoidal probability monad $(\mc{P}, \delta_\bullet, \Rbag, \otimes)$ on $\mathsf{QCB}_{h}$, the \emph{Baire probability monad}.
\end{theorem}

\subsection{Relation to the Radon monad on compact Hausdorff spaces}\label{sec:relationToRadonMonad}

The monad $\mc{P}$ restricts to the \emph{Radon monad} on compact Hausdorff spaces (see \cite[Section 6]{van2021probability} for a detailed definition of the Radon monad, and \cite{swirszcz1974monadic} for the original account). This is the content of the following theorem.

\begin{theorem}\label{convergence_of_measures_on_compactum}
    Let $X$ be a compact Hausdorff space. Then every Baire measure on $X$ is $k$-regular (so that $\mathcal{M}(X)$ consists exactly of the Baire measures on $X$). Moreover, the topology of $\mathcal{P}(X)$ coincides with the topology of weak convergence of measures, and $\mathcal{P}(X)$ is a compact Hausdorff space. 
\end{theorem}
\begin{proof}
    The fact that every Baire measure on $X$ is $k$-regular was already shown in Lemma \ref{lemma:radonMeasuresKregular}. For the second part of the statement, notice that, since $X$ is compact, $C_b(X)=C(X)$ carries the uniform topology, by Lemma \ref{lemma:YtoTheXforYmetricSpace}. Moreover, $\mc{P}(X) \subseteq C(X)^\wedge $ is closed in the unit ball of the natural dual space $C(X)^\wedge$ (see Definition \ref{definition:naturalDual}), on which its topology coincides with the weak-$*$ topology (Lemma \ref{lemma:unitBallOfNaturalDualCarriesWeakStarTop}). Therefore, $\mc{P}(X)$ is compact (by Alaoglu's theorem) and its topology coincides with the topology of weak convergence of measures.
\end{proof}

\subsection{Relation to the Giry monad on Polish spaces}\label{sec:relationToGiryMonad}

In this section, we show that the Baire probability monad $\mc{P}$ on $\mathsf{QCB}_{h}$ restricts to the classical Giry monad on Polish spaces. We begin by relating convergence in $\mc{P}(X)$ to weak convergence of measures.

\begin{theorem}\label{theorem:convergenceInPXvsWeakConvergence}
    Let $X$ be a Polish space. Then a sequence $(\mu_n)$ of probability measures converges in $\mathcal{P}(X)$ if, and only if, it converges weakly.  
\end{theorem}
\begin{proof}
    Since $\mc{P}(X)$ is, by definition, a subspace of $\mc{M}(X)=C_b(X)^\wedge$, Lemma \ref{lemma:ConvergenceInYtoTheX} implies that a sequence $(\mu_n)$ converges in $\mc{P}(X)$ if, and only if, it converges uniformly on compact subsets of $C_b(X)$.
    Hence, the ``if'' direction is immediate. 
    
    For the converse, suppose that $\mu_n \to \mu$ weakly. We want to show that $\mu_n \to \mu$ uniformly on compact subsets of $C_b(X)$. So let $L\subseteq C_b(X)$ be compact and $\epsilon >0$. We collect a number of observations which we then combine to obtain the required estimate.
    
    1. By Prohorov's theorem \cite[p.~202, Theorem 8.6.2]{bogachev2007measure}, the family $\{|\mu_n-\mu|\}$ is uniformly tight. This means that there exists a compact subset $K_\epsilon\subseteq X$ such that for all $n\in \mathbb{N}$, 
    \begin{equation}\label{equation:uniformlyTight}
        |\mu_n-\mu|(X\mathbin{\backslash} K_\epsilon) < \epsilon.
    \end{equation}
    
    2. Since the family of functions $L\subseteq C_b(X)$ is compact, it is uniformly bounded by some constant $C>0$. 
    
    3. The restrictions $\mu_n|_{K_\epsilon}, \mu|_{K_\epsilon}$ are Radon measures on the compact Hausdorff space $K_\epsilon$, and $\mu_n|_{K_\epsilon} \to \mu|_{K_\epsilon}$ weakly. By Lemma \ref{lemma:unitBallOfNaturalDualCarriesWeakStarTop}, the weak-$*$ topology coincides with the topology of compact convergence on the unit ball of the dual of $C(K_\epsilon)$. Hence, the sequence of functionals $(\mu_n|_{K_\epsilon})$ converges uniformly on compact subsets of $C(K_\epsilon)$ to $\mu|_{K_\epsilon}$. 
    
    4. The image of $L$ under restriction of functions to $K_\epsilon$,
        $$ L|_{K_\epsilon} := \{f|_{K_\epsilon} \mid f \in L \} \subseteq C(K_\epsilon),$$
    is compact, since it is the image of the compact set $L$ under continuous map $f \mapsto f \circ \iota_{K_\epsilon}$, where $\iota_{K_\epsilon}: K_{\epsilon} \to X $ is the inclusion map. 
    
    5. Since $(\mu_n|_{K_\epsilon})$ converges uniformly on compact subsets of $C(K_\epsilon)$ and $L|_{K_\epsilon}$ is compact, we have that for sufficiently large $n$, 
    \begin{equation}\label{equation:compactConvergence}
        \sup_{f \in L} \Big| \int_{K_\epsilon} f \:\mathrm{d} \mu_n- \int_{K_\epsilon} f \:\mathrm{d} \mu \Big| = \sup_{f \in L|_{K_\epsilon}} \Big| \int f \:\mathrm{d} \mu_n|_{K_\epsilon} - \int f \:\mathrm{d} \mu|_{K_\epsilon} \Big| < \epsilon.
    \end{equation}
    
    Combining these observations, we obtain that for all $f\in L$ and sufficiently large $n$, 
    \begin{align*}
        \Big| \int f \:\mathrm{d} \mu_n - \int f \:\mathrm{d} \mu \Big| 
        &\leq \int_{X\mathbin{\backslash} K_\epsilon} |f| \:\mathrm{d} |\mu_n - \mu| + \Big| \int_{K_\epsilon} f \:\mathrm{d} \mu_n - \int_{K_\epsilon} f \:\mathrm{d} \mu \Big| \\
        &\leq  C \epsilon + \Big| \int_{K_\epsilon} f \:\mathrm{d} \mu_n - \int_{K_\epsilon} f \:\mathrm{d} \mu \Big| \tag{Eq.~\eqref{equation:uniformlyTight}, 2.}\\
        &\leq (C+1) \,\epsilon,  \tag{Eq.~\eqref{equation:compactConvergence}}
    \end{align*}
    which is what we wanted to show.
\end{proof}

Using this, we now show that probability monad $\mathcal{P}$ restricts to the Giry monad on Polish spaces.

\begin{theorem}\label{theorem:radonMeasuresPolishkRegular}
    Let $X$ be a Polish space. Then every Baire measure is a $k$-regular measure and hence, the notions of Baire, Radon and $k$-regular measure all coincide (see \ref{theorem:baireMeasuresVsRadonMeasures}). Moreover, the topology on $\mathcal{P}(X)$ is given by the topology of weak convergence of measures and $\mathcal{P}(X)$ is again a Polish space.
\end{theorem}
\begin{proof}
    That every Baire measure is $k$-regular was already shown in Lemma \ref{lemma:radonMeasuresKregular}. Now, let $\mathcal{P}_w(X)$ be the space of Baire (equivalently, Radon, by Theorem \ref{theorem:baireMeasuresVsRadonMeasures}) measures on $X$ with the topology of weak convergence. As sets, $\mathcal{P}_w(X)=\mathcal{P}(X)$. We want to show that the respective topologies coincide, as well. First, note that the identity map $\mathcal{P}(X) \to \mathcal{P}_w(X)$ is continuous, since a convergent net in $\mathcal{P}(X)$ is also weakly convergent. By \cite[p.~213, Theorem 8.9.4]{bogachev2007measure}, our assumption that $X$ is a Polish space implies that $\mathcal{P}_w(X)$ is a Polish space, too. Hence, sequential continuity of the identity map $\mathcal{P}_w(X) \to \mathcal{P}(X)$, which is provided by Theorem \ref{theorem:convergenceInPXvsWeakConvergence}, implies continuity, so the two topologies coincide and $\mathcal{P}(X)=\mathcal{P}_w(X)$ also as topological spaces.
\end{proof}

\subsection{The Baire probability monad is strongly affine}\label{sec:BaireMonadStronglyAffine}

A basic phenomenon in probability theory is that a deterministic random variable is independent of any other random variable. A measure-theoretic way to express this is the following. If one of the marginals of a probability measure on a product is a point mass, then it is the product of its marginals. Curiously, this \emph{deterministic marginal independence} can \emph{fail} for quasi-Borel spaces \cite[Section 3]{fritz2023dilations}. Such (perhaps counter-intuitive) behaviour cannot occur when working with weakly Hausdorff QCB spaces: the Baire probability monad on $\mathsf{QCB}_{h}$ is \emph{strongly affine} \cite[Definition 1]{jacobs2016affine}. 

\begin{definition}
    Let $T$ be a strong monad on a category $\mathsf{C}$ with finite products. Then $T$ is \emph{strongly affine} if for all objects $X,Y$ of $\mathsf{C}$, the following diagram is a pullback square:
    % https://q.uiver.app/#q=WzAsNCxbMCwwLCJUKFgpXFx0aW1lcyBZIl0sWzEsMCwiWSJdLFswLDEsIlQoWFxcdGltZXMgWSkiXSxbMSwxLCJUKFkpIl0sWzAsMSwiXFxwaV8yIl0sWzAsMiwicyIsMl0sWzIsMywiVChcXHBpXzIpIiwyXSxbMSwzLCJcXGV0YSJdXQ==
    \[\begin{tikzcd}
	{T(X)\times Y} & Y \\
	{T(X\times Y)} & {T(Y)}
	\arrow["{\pi_2}", from=1-1, to=1-2]
	\arrow["s"', from=1-1, to=2-1]
	\arrow["{T(\pi_2)}"', from=2-1, to=2-2]
	\arrow["\eta", from=1-2, to=2-2]
    \end{tikzcd}\]
    Here, $\eta$ is the unit of $T$ and $s$ is the right strength. 
\end{definition}

For the Baire probability monad $\mc{P}$ on $\mathsf{QCB}_{h}$, this will come down to the aforementioned phenomenon of deterministic marginal independence, as we now show. The proof relies heavily on the fact that product measures are well-behaved on QCB spaces (they agree with the usual measure-theoretic ones, see Lemma \ref{lemma:productMeasuresOnQCBspaces}); this is not the case for general CGWH spaces.

\begin{theorem}
    The probability monad $\mathcal{P}$ on weakly Hausdorff QCB spaces is strongly affine. 
\end{theorem}
\begin{proof}
    Let $X$ and $Y$ be weakly Hausdorff QCB spaces. We have to show that the following diagram is a pullback square:
    % https://q.uiver.app/#q=WzAsNCxbMCwwLCJcXG1hdGhjYWx7UH0oWClcXHRpbWVzIFkiXSxbMSwwLCJZIl0sWzEsMSwiXFxtYXRoY2Fse1B9KFkpIl0sWzAsMSwiXFxtYXRoY2Fse1B9KFhcXHRpbWVzIFkpIl0sWzAsMSwiXFxwaV8yIl0sWzEsMiwiXFxkZWx0YV9cXGJ1bGxldCJdLFszLDIsIihcXHBpXzIpXyoiLDJdLFswLDMsIi1cXG90aW1lcyBcXGRlbHRhX3stfSIsMl1d
    \[\begin{tikzcd}
	{\mathcal{P}(X)\times Y} & Y \\
	{\mathcal{P}(X\times Y)} & {\mathcal{P}(Y)}
	\arrow["{\pi_2}", from=1-1, to=1-2]
	\arrow["{\delta_\bullet}", from=1-2, to=2-2]
	\arrow["{(\pi_2)_*}"', from=2-1, to=2-2]
	\arrow["{-\:\otimes \:\delta_{-}}"', from=1-1, to=2-1]
    \end{tikzcd}\]
    By how pullbacks are constructed in $\mathsf{QCB}_{h}$ (equivalently, $\cgwh$), this is equivalent to showing that the canonical map,
        $$ \Phi: \mathcal{P}(X) \times Y \to \{(\mu, y) \in \mc{P}(X\times Y) \times Y \mid (\pi_2)_*\mu = \delta_y \}, \;\;(\nu, y) \mapsto (\nu \otimes \delta_y, y), $$
    is an isomorphism (where the codomain is equipped with the subspace topology). Note first that $\Phi$ has a continuous left inverse, given by $(\mu, y) \mapsto ((\pi_1)_*\mu, y)$. Hence, it suffices to verify that $\Phi$ is surjective. Let $\mu \in \mc{P}(X\times Y)$ and $y\in Y$ such that $(\pi_2)_*\mu = \delta_y$. We want to show that there exists a $\nu\in \mc{P}(Y)$ such that $\mu = \nu \otimes \delta_y$. Our claim is now that one may take $\nu := (\pi_1)_*\mu$, i.e. that $\mu$ is the product of its marginals. Since $X$ and $Y$ are QCB spaces, the Baire $\sigma$-algebra on the product coincides with the product of the Baire $\sigma$-algebras on the factors (see Lemma \ref{lemma:productMeasuresOnQCBspaces}):
        $$\Ba(X\times Y) = \Ba(X) \otimes \Ba(Y) = \sigma(\{ A \times B \mid A \in \Ba(X), B\in \Ba(Y)\}). $$
    Therefore, our claim reduces to showing that for all $A \in \Ba(X), B\in \Ba(X)$, 
    \begin{equation}\label{eq:nonintertwined}
        \mu(A \times B) = \nu(A) \cdot \delta_y(B).
    \end{equation}
    The proof is now exactly as in the case of the Giry monad given in \cite[Example 1]{jacobs2016affine}. For the convenience of the reader, we reproduce it here. There are two cases:
    \begin{enumerate}
        \item \emph{Suppose that $y\not\in B$.} We show that both sides of \eqref{eq:nonintertwined} vanish. For the right hand side, this is immediate, while for the left hand side, this follows from the monotonicity of $\mu$:
            $$ \mu(A \times B) \leq \mu(X \times B) = \delta_y(B) = 0. $$
        \item \emph{Suppose that $y \in B$.} Then by what we have just shown,
            $ \mu(A \times (X \setminus B)) = 0, $ and hence,
        $$\mu(A \times B) = \mu(A \times B) +  \mu(A \times (X \setminus B)) 
                            = \mu(A \times Y) = \nu(A) \cdot \delta_y(B).
        $$
    \end{enumerate}
    This completes the proof.
\end{proof}

\bibliographystyle{alpha}
\bibliography{bibliography}

\end{document}